\documentclass[11pt,a4paper,openany]{amsart}
\usepackage{mathrsfs}

\usepackage{amsmath,amsfonts,verbatim}
\usepackage{latexsym}
\usepackage{amssymb,leftidx}
\usepackage{overpic}
\usepackage{color}
\usepackage{epsfig}
\usepackage{subfigure}
\usepackage{tikz}

\newcommand\R{\mathbb{R}}
\newcommand\C{\mathbb{C}}

\newcommand\LL{\mathcal{L}}
\newcommand\CC{\mathcal{C}}

\numberwithin{equation}{section}

\newtheorem{lemma}{Lemma}[section]
\newtheorem{theorem}{Theorem}[section]

\newtheorem{remark}{Remark}[section]

\begin{document}
\title[Uniform resolvent and Sobolev estimates]{Uniform resolvent estimates for Schr\"odinger operator with an inverse-square potential}

\author{Haruya Mizutani}
\address{Department of Mathematics, Graduate School of Science, Osaka University, Toyonaka, Osaka 560-0043, Japan.}
\email{haruya@math.sci.osaka-u.ac.jp}

\author{Junyong Zhang}
\address{Department of Mathematics, Beijing Institute of Technology, Beijing 100081, China;  Department of Mathematics, Cardiff University, UK.}
\email{zhang\_junyong@bit.edu.cn; ZhangJ107@cardiff.ac.uk}

\author{Jiqiang Zheng}
\address{Institute of Applied Physics and Computational Mathematics, Beijing100088, China.}
\email{zhengjiqiang@gmail.com}

\maketitle


\begin{abstract}
We study the uniform resolvent estimates for Schr\"odinger operator with a Hardy-type singular potential.
 Let $\mathcal{L}_V=-\Delta+V(x)$ where $\Delta$ is the usual Laplacian on $\R^n$ and $V(x)=V_0(\theta) r^{-2}$ where $r=|x|, \theta=x/|x|$ and $V_0(\theta)\in\CC^1(\mathbb{S}^{n-1})$ is a real function such that the operator $-\Delta_\theta+V_0(\theta)+(n-2)^2/4$ is a strictly positive operator on $L^2(\mathbb{S}^{n-1})$. We prove some new uniform weighted resolvent estimates and also obtain some uniform Sobolev estimates associated with the operator $\mathcal{L}_V$.

\end{abstract}

\begin{center}
 \begin{minipage}{120mm}
 \end{minipage}
 \end{center}

\maketitle 

\section{Introduction and main results}
In this paper, we study the uniform resolvent estimates and their applications to the Sobloev inequalities and to the global-in-time inhomogeneous Strichartz estimates with non-admissible pairs. Consider the Schr\"odinger operator
\begin{equation}\label{ope}\LL_V=-\Delta+V(x)\end{equation}
on $L^2(\R^n)$ with $n\geq 3$ where the operator  $\Delta$ is the usual Laplacian on $\R^n$ and the potential $V(x)=V_0(\theta) r^{-2}$ with $r=|x|, \theta=x/|x|$ and $V_0(\theta)\in\CC^1(\mathbb{S}^{n-1})$ is a real function.
The inverse-square potential is a typical example of critical decaying potentials, which is on a borderline for the validity of the resolvent and Strichartz estimates; we refer to \cite{D,GVV}. \vspace{0.2cm}

This paper is motivated by recent work of Bouclet and the first author \cite{BM1} and the first author \cite{Miz} in which the effect of  decaying potentials in uniform resolvent estimates and
global-in-time Strichartz estimates were investigated. In \cite{BM1}, the weighted resolvent estimates $\|w(\LL_V-z)^{-1}w^*\|_{L^2\to L^2}$ uniformly in $z$ were proved to hold with $w$ being a large class of weight functions in Morrey-Campanato spaces.
The full set of global-in-time Strichartz estimates including the endpoint case was also obtained in
 \cite{BM1}, but non-admissible inhomogeneous cases were not considered there.
 The class of potentials we consider here includes the inverse-square type potentials.
In \cite{Miz}, the uniform Sobolev estimates for the resolvent were proved under the assumption that zero energy is neither an eigenvalue nor a resonance in a suitable sense for the operator $\LL_V$.
The first author also proved global-in-time inhomogeneous Strichartz estimates hold for some non-admissible pairs. But one needs the requirement that $V\in L^{\frac n2}(\R^n)$ with $n\geq3$ which is not satisfied by the
inverse-square potential. In light of this observation, the purpose of this paper is to study the uniform resolvent estimates, the Sobolev inequalities and the non-admissible inhomogeneous Strichartz estimates which are associated with Schr\"odinger operator with 
an inverse-square decaying potential. \vspace{0.2cm}

The uniform resolvent estimates play a fundamental role in the establishment of time-decay estimates or Strichartz estimates, see \cite{JSS, Kato, RS}.
When the potential $V$ is smooth enough and decays sufficiently fast at infinity,  for example $V$ belongs to  Kato class (see \cite{RS}),  there is a number of  literature on the resolvent estimates of the Schr\"odinger operator with potentials and their applications to global-in-time dispersive estimates, such as time-decay estimates, or Strichartz estimates, in the past decades; see e.g. \cite{GS, JK,RT}  for the resolvent estimates; \cite{BD,BG,DFVV, G, FFFP} for the dispersive and Strichartz estimates and the references therein.\vspace{0.2cm}

In this paper, as mentioned above, we focus on the Schr\"odinger operator $\mathcal L_V$ given in \eqref{ope} which appears frequently in mathematics and physics. The study of the operator is connected with the combustion theory to the Dirac equation with Coulomb potential, and the study of perturbations of classic space-time metrics such as Schwarzschild and Reissner--Nordstr\"om; see  \cite{BPSS,BPST,PSS, PSS1,KSWW,VZ} and the references therein.

The Strichartz estimates and time-decay estimates for the dispersive equations with an inverse-square potential were  studied in \cite{BPSS, BPST, PSS, PSS1}.
In particular, Burq et al. \cite{BPST} established the weighted uniform resolvent estimate
\begin{equation}\label{uni-resolvent}
\||x|^{-1}(\LL_V-\sigma)^{-1}|x|^{-1}\|_{L^{2}(\R^n)\to L^2(\R^n)}\leq C,
\end{equation}
and then they used it to prove the full set of the Strichartz estimates excluding the double-endpoint inhomogeneous estimates which were proved in \cite{BM1} later.
To prove the inhomogeneous Strichartz estimates for non-admissible pairs and to obtain more Sobolev inequality, the above uniform resolvent estimate \eqref{uni-resolvent} is not enough. For our purpose, we have to generalize \eqref{uni-resolvent} to \eqref{in-local-sm1}
stated below in our first result.\vspace{0.2cm}

Before stating our first result, we introduce some notation.  Let $\nu_0>0$ be the positive square root \footnote{ To ensure $\nu_0>0$, it is enough to choose $V_0(\theta)$ such that $-\Delta_\theta+V_0(\theta)+(n-2)^2/4$ is a strictly positive operator on $L^2(\mathbb{S}^{n-1})$. For example, one can take $V_0(\theta)\geq a$ where $a>-(n-2)^2/4$ to guarantee $\nu_0>0$. } of the smallest eigenvalue of the operator
 $-\Delta_\theta+V_0(\theta)+(n-2)^2/4$ where $\Delta_\theta$ is the usual Laplacian on the sphere $\mathbb{S}^{n-1}$.  We define the interval $R_{\nu_0}\subset \R$ depending on $\nu_0$ by
 \begin{equation}\label{Rv} R_{\nu_0}=
 \begin{cases} (\frac12,\frac32), \quad \nu_0>1/2;\\
 (1-\frac{\nu_0^2}{1-2\nu_0^2}, 1+\frac{\nu_0^2}{1-2\nu_0^2}),\quad 0<\nu_0\leq1/2.
 \end{cases}
 \end{equation}

\begin{theorem}[Weighted resolvent estimates]\label{thm:w-resolvent}
Let $n\geq 3$ and let $\mathcal{L}_V$ be the operator on $L^2(\R^n)$ in \eqref{ope}.  Suppose the real function $V_0(\theta):=r^2V(x)\in\CC^1(\mathbb{S}^{n-1})$ and
the smallest eigenvalue of the operator $-\Delta_\theta+V_0(\theta)+(n-2)^2/4$ on $L^2(\mathbb{S}^{n-1})$ is $\nu_0^2>0$.  Let $\alpha\in R_{\nu_0}$ be defined in \eqref{Rv}.
Then  there exists a constant $C$ such that the uniform weighted resolvent estimates
hold
\begin{equation}\label{in-local-sm1}
\sup_{\sigma\notin \R^+} \|r^{-\alpha}(\LL_V-\sigma)^{-1}r^{-2+\alpha}f\|_{L^{2}(\R^n)}\leq C\|f\|_{L^{2}(\R^n)},\quad f\in\mathcal{C}_0^\infty(\R^n).
\end{equation}
\end{theorem}
\begin{remark} This is a generalization of \cite [Theorem 2.1]{BPST} in which they proved \eqref{in-local-sm1} with $\alpha=1$ .
The smallest eigenvalue $\nu_0^2$ plays an important role in \eqref{in-local-sm1}.
\end{remark}

\begin{remark} Let $\LL_V=-\Delta_g+V$ be defined  on a manifold and $\langle x\rangle=(1+|x|^2)^{1/2}$.  On the asymptotically Euclidean space, Bony-H\"afner \cite{BH} proved the resolvent estimates
at low frequency
\begin{equation*}
\|\langle x\rangle^{-\alpha}(\mathcal{L}_V-\sigma)^{-1} \langle x\rangle^{-\beta}\|_{L^2\to L^2}\leq C, \quad |\sigma|\leq1
\end{equation*}
provided $\alpha,\beta>1/2$ and $\alpha+\beta>2$ when $V=0$.
On the asymptotically conic manifold, Bouclet-Royer \cite{BR} showed the sharp resolvent estimate at low frequency
\begin{equation*}
\|\langle x\rangle^{-1}(\mathcal{L}_V-\sigma)^{-1} \langle x\rangle^{-1}\|_{L^2\to L^2}\leq C, \quad |\sigma|\leq 1
\end{equation*}
 when $V=0$. The last two authors \cite{ZZ1} extended this estimate with a decaying $O(\langle x\rangle^{-2})$ potential such that the operator $\LL_V$ has no nonpositive eigenvalues or zero-resonance.
 The result here is on Euclidean space but with flexible weights such as $|x|^{-\alpha}$ and also includes the high frequency estimates.
\end{remark}

\begin{remark} One can use the same argument to derive the similar resolvent estimates \eqref{in-local-sm1} on a metric cone  as the last two authors did in \cite{ZZ}. 
It would be interesting to show a similar result of Theorem \ref{thm-in} below for Schr\"odinger operator $\LL_V$ on the metric cone, for which the last two authors proved the Strichartz estimates in \cite{ZZ,ZZ2}.
But there is an obstacle to obtain \eqref{unf-sob} below on the metric cone since the metric of section cross is so general that the conjugated points could appear. The difficulties arise from the conjugated points.
\end{remark}

When $V\equiv0$, the following uniform Sobolev inequality was proved by Kenig-Ruiz-Sogge \cite{KRS} and Guti\'errez \cite{Gut}:
\begin{equation}
\label{uniform_Sobolev}
\left\|(-\Delta-\sigma)^{-1}f\right\|_{L^{q,2} (\R^n)}\leq
C|\sigma|^{\frac n2(\frac1p-\frac1q)-1}\|f\|_{L^{p,2}(\R^n)},\quad \sigma\notin \R^+,\ f\in\mathcal{C}_0^\infty(\R^n),
\end{equation}
where $n\ge3$ and $(p,q)$ satisfies
\begin{align}
\label{p_q_0}
\frac{2}{n+1}\le \frac1p-\frac1q\le \frac2n,\ \frac{2n}{n+3}<p<\frac{2n}{n+1},\ \frac{2n}{n-1}<q<\frac{2n}{n-3},
\end{align}
and $L^{q,r}(\R^n)$ is the usual Lorentz space. Precisely speaking, they proved \eqref{uniform_Sobolev} with $L^{p,2},L^{q,2}$ replaced by $L^p,L^q$, respectively. However, \eqref{uniform_Sobolev} is an immediate consequence of their results and real interpolation theory. Note that the condition \eqref{p_q_0} is known to be sharp (see \cite{Gut}). It is also worth noting that the uniform Sobolev inequality is a powerful tool in spectral and scattering theory for Schr\"odinger equations (see \cite{IoSc, KRS}), as well as nonlinear elliptic equations such as the Ginzburg-Landau equation (see \cite{Gut}).

As a second result, we extend \eqref{uniform_Sobolev} to the operator $\LL_V$. Let us set $$
\mu_0=
\begin{cases}
1/2,&\nu_0\ge1/2;\\\frac{\nu_0^2}{1-2\nu_0^2},&0<\nu_0<1/2.
\end{cases}
$$

\begin{theorem}[Uniform Sobolev inequality]\label{thm-unf-sob} Let $\mathcal L_V$ be given as above and suppose  \begin{equation}\label{p_q}
\frac{2}{n+1}\le \frac1p-\frac1q\le \frac2n,\ \frac{2n}{n+2(1+\mu_0)}<p<\frac{2n}{n+1},\ \frac{2n}{n-1}<q<\frac{2n}{n-2(1+\mu_0)}.
\end{equation}
Then there exists a positive constant $C$ such that
\begin{equation}\label{unf-sob}
\left\|(\LL_V-\sigma)^{-1}f\right\|_{L^{q,2} (\R^n)}\leq
C|\sigma|^{\frac n2(\frac1p-\frac1q)-1}\|f\|_{L^{p,2}(\R^n)},\ \sigma\notin \R^+,\ f\in\mathcal{C}_0^\infty(\R^n).
\end{equation}
\end{theorem}

 \begin{remark}
When $\nu_0\ge 1/2$, \eqref{p_q} coincides with \eqref{p_q_0} (see Figure \ref{figure_1} below) in which case Theorem \ref{thm-unf-sob} gives the full range of uniform Sobolev inequalities for $\mathcal L_V$. Uniform Sobolev inequalities for Schr\"odinger operators have been recently studied in several papers. Bouclet and the first author \cite{BM1} and the first author \cite{Miz1} showed \eqref{unf-sob} for $\LL_V$ under \eqref{p_q} and $1/p+1/q=1$. For the special case $(p,q)=(\frac{2n}{n+2},\frac{2n}{n-2})$, Guillarmou and Hassell \cite{GH} showed such estimates to the Laplace operator on nontrapping asymptotically conic manifolds, and Hassell and the second author \cite{HZ} extended it to potential perturbations with smooth potentials decaying at infinity like $\langle x\rangle^{-3}$ and without 0 resonance or eigenvalue. Compared with these results, we here prove more results ($p,q$ may not be dual each other) on $\R^n$ for potentials with weaker decay at infinity and critical singularity at the origin.
 \end{remark}

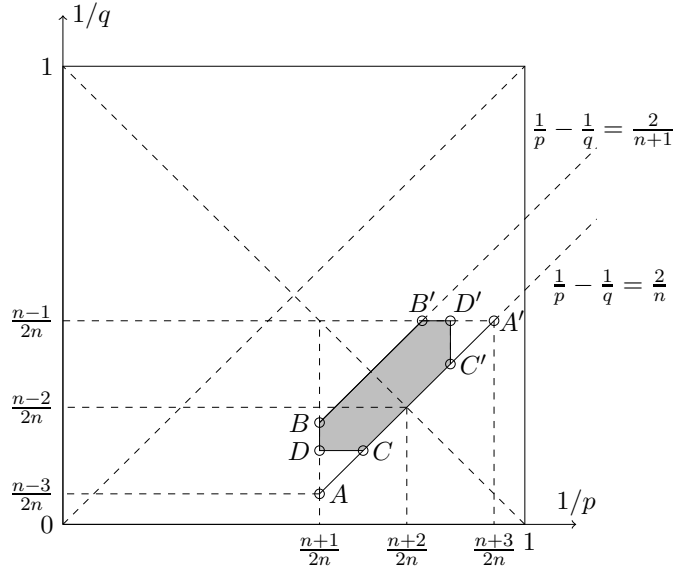
\begin{figure}[htbp]
\label{figure_1}
\begin{center}
\scalebox{0.9}[0.9]{
\begin{tikzpicture}

\draw (0,0) rectangle (6.75,6.75);
\draw[->]  (0,0) -- (0,7.5);
\draw[->]  (0,0) -- (7.5,0);
\filldraw[fill=gray!50](4.3875,1.0875)--(3.75,1.0875)--(3.75,1.5)--(5.25,3)--(5.6625,3)--(5.6625,2.3625); 
\draw (6.75,0) node[below] {$\ 1$};
\draw (0,6.75) node[left] {$1$};
\draw (0,0) node[below, left] {$0$};
\draw (7.5,0) node[above] {$1/p$};
\draw (0,7.5) node[right] {$1/q$};
\draw[dashed] (0,0) -- (6.75,6.75);
\draw[dashed] (6.75,0) -- (0,6.75);
\draw[dashed] (6.3,3) -- (7.8,4.5);
\draw (7,3.5) node[right] {$\frac1p-\frac1q=\frac 2n$};
\draw[dashed] (5.25,3) -- (7.8,5.55);
\draw (7.9,5.4) node[above] {$\frac1p-\frac1q=\frac{2}{n+1}$};
\draw[dashed] (3.75,1.5) -- (3.75,3.0) -- (5.25,3.0);
\draw[dashed] (3.75,0.45) -- (0,0.45);
\draw[dashed] (6.3,3) -- (6.3,0);
\draw (3.75,0.45) circle (2pt) node[right] {$A$};    
\draw (6.3,3)  circle (2pt) node[right] {$\!A'$}; 
\draw (3.75,1.5)  circle (2pt) node[left] {$B$};   
\draw (5.25,3)  circle (2pt) node[above] {$B'$}; 
\draw (4.3875,1.0875)  circle (2pt) node[right] {$C$}; 
\draw (5.6625,2.3625) circle (2pt) node[right] {$C'$};    
\draw (3.75,1.0875) circle (2pt) node[left] {$D$};    
\draw (5.6625,3) circle (2pt) node[above] {$\ \ \ \ D'$};    
\draw (3.75,0.45) -- (6.3,3); 
\draw[dashed] (3.75,0.45) -- (3.75,1.5) ; 
\draw[dashed] (5.25,3) -- (6.3,3) ; 
\draw (3.75,1.5) -- (5.25,3); 
\draw[dashed] (3.75,0.45) -- (3.75,0);
\draw (3.75,0) node[below] {$\frac{n+1}{2n}$};
\draw[dashed] (5.025,1.725) -- (5.025,0) node[below] {$\frac{n+2}{2n}$};
\draw[dashed] (5.6625,2.3625) -- (5.6625,3);
\draw(6.3,0) node[below]  {$\frac{n+3}{2n}$};
\draw (0,3) node[left] {$\frac{n-1}{2n}$};
\draw[dashed] (5.025,1.725) -- (0,1.725) node[left] {$\frac{n-2}{2n}$};
\draw (0,0.45) node[left] {$\frac{n-3}{2n}$};
\draw[dashed] (0,3) -- (3.75,3.0);
\end{tikzpicture}
}
\end{center}
\caption{The condition \eqref{p_q_0} corresponds to the trapezium $ABB'A'$ with two closed line segments $\overline{AB}$, $\overline{B'A'}$ removed, while the condition \eqref{p_q} with $\mu_0<1/2$ corresponds to the shaded region surrounded by the polygon $CDBB'D'C'$ with 4 closed line segments $CD$, $DB$, $B'D'$ and $D'C'$ removed. Here $A=(\frac{n+1}{2n},\frac{n-3}{2n})$, $B=(\frac{n+1}{2n},\frac{n^2-3n+1}{2n(n+1)})$, $C=(\frac{n+2\mu_0}{2n},\frac{n-2(1+\mu_0)}{2n})$, $D=(\frac{n+1}{2n},\frac{n-2(1+\mu_0)}{2n})$ and $A',B',C',D'$ are dual points of $A,B,C,D$, respectively.}
\end{figure}

Finally we  state the result about inhomogeneous Strichartz estimates for non-admissible pairs.
Before stating the result, we recall the background of the Strichartz estimates without potential. Consider the Cauchy problem for the  inhomogeneous Schr\"odinger equation
\begin{equation}
\begin{cases}
i\partial_t u+\Delta u=F(t,x),\quad t\in\R, x\in\R^n;\\
u(0)=u_0(x).
\end{cases}
\end{equation}
By Duhamel's formula, the solution $u$ is given by
\begin{equation}
u(t)=e^{it\Delta}u_0-i\int_0^t e^{i(t-s)\Delta} F(s)ds.
\end{equation}
R. Strichartz \cite{Str} in 1977 proved that there exists a constant $C$ such that
\begin{equation}
\|u(t)\|_{L^q(\R;L^r(\R^n))}\leq C\left(\|u_0\|_{L^2}+\|F\|_{L^{\tilde{q}'}(\R; L^{\tilde{r}'}(\R^n))}\right)
\end{equation}
with $q=r=\tilde{q}=\tilde{r}=2(n+2)/n$ when $u_0\in L^2(\R^n),  F\in L^{\tilde{q}'}(\R; L^{\tilde{r}'}(\R^n))$. From then,
there are many works devoted to this type of a priori estimates, so called the \emph{Strichartz estimate}, for solutions to the Schr\"odinger equation
in which $q$ is possibly not equal to the exponent $r$; we refer the readers to \cite{GV, KT} and the references therein.  The Strichartz estimates have been used to prove rich results on the
well-posed theory and nonlinear scattering theory for the semi-linear Schr\"odinger equations on Euclidean space, for example, see \cite{GV,Tbook} and the references therein.\vspace{0.2cm}

In particular, if $F=0$,  the Strichartz estimate becomes
\begin{equation}\label{h-str}
\left\|e^{it\Delta}u_0\right\|_{L^q(\R;L^r(\R^n))}\leq C\|u_0\|_{L^2}
\end{equation}
and if $u_0=0$, then
\begin{equation}\label{inh-str}
\left\|\int_0^t e^{i(t-s)\Delta} F(s)ds\right\|_{L^q(\R;L^r(\R^n))}\leq C\|F\|_{L^{\tilde{q}'}(\R; L^{\tilde{r}'}(\R^n))}.
\end{equation}
The first one is known as a \emph{homogeneous Strichartz estimate} and the second one is called  \emph{inhomogeneous Strichartz estimate}.
If $(q,r)$ satisfies
\begin{equation}\label{admissible}
q,r\in [2,\infty], \quad 2/q=n(1/2-1/r), \quad (q,r,n)\neq (2,\infty, 2),
\end{equation} we say $(q,r)$ is a Schr\"odinger admissible pair, denoted by $(q,r)\in \Lambda_0$.  From \cite{KT}, the homogeneous estimate \eqref{h-str} holds if and only if  $(q,r)\in \Lambda_0$. But there are some differences for the inhomogeneous estimates.
It has been known that if both $(q,r)$ and $(\tilde{q},\tilde{r})$ are admissible pairs, the inhomogeneous estimate \eqref{inh-str} holds.
Furthermore, it is known that there exist the exponent pairs $(q,r)$ and $(\tilde{q},\tilde{r})$ which do not satisfy the admissible condition, but the inhomogeneous estimate can still be valid; we refer the reader to
 T. Cazenave and F. Weissler \cite{CW} and T. Kato \cite{Kato} for Schr\"odinger and to  Harmse \cite{Harmse} and Oberlin \cite{Oberlin} for wave with $q=r$.  After that,
 D. Foschi \cite{Foschi} and M. Vilela \cite{Vilela} independently and greatly extended the range of the exponent pairs $(q,r)$ and $(\tilde{q},\tilde{r})$ for
 which the inhomogeneous Strichartz estimate holds. R. J. Taggart \cite{Taggart} generalized the inhomogeneous Strichartz estimate in an abstract mechanism. For more results on the inhomogeneous Strichartz estimate,
 we refer to Y. Koh \cite{Koh} and R. Schippa \cite{Schippa}. However, the problem of finding all possible exponents pairs $(q,r)$ such that the inhomogeneous estimate \eqref{inh-str} is available remains open. \vspace{0.2cm}

 It is worth remarking here that  the argument is based on the method introduced in Keel-Tao \cite{KT} and
 most of the inhomogeneous Strichartz estimates are established there under the assumption that the propagator satisfies the energy estimate
 \begin{equation}\label{L2-est}
 \|U(t)\|_{L^2\to L^2}\leq C
 \end{equation}
  and the dispersive estimate
   \begin{equation}\label{dis-est}
 \|U(t)U^*(s)\|_{L^1\to L^\infty}\leq C|t-s|^{-\sigma}, \quad t\neq s.
 \end{equation}
 In particular, for the Sch\"odinger operator without potential, $U(t)=e^{it\Delta}$ and $\sigma=n/2$. It is known that the Strichartz estimate still holds when the pairs $(q,r)$ and $(\tilde{q},\tilde{r})$ are admissible pairs
 even though the dispersive estimate \eqref{dis-est} fails. For example, Burq et al. \cite{BPSS} proved the Strichartz estimates for the operator $-\Delta+a|x|^{-2}$ on $\R^n$ with $a>-(n-2)^2/4$ and $n\geq3$,  but the dispersive estimate fails due to the negative inverse-square potential, e.g. see \cite{FFFP,PSS1}; and the
 Strichartz estimates including endpoints still hold on non-trapping asymptotically conic manifold or in a conic space (see \cite{HZ, ZZ,ZZ2}) 
 but the dispersive estimate fails due to the conjugated points (e.g. see \cite{HW}).  In the light of those Strichartz estimates were proved for admissible pairs
 even without the dispersive estimate, it is natural to ask whether the inhomogeneous Strichartz estimates hold for some non-admissible pairs.
Due to the inverse-square potential, the usual dispersive estimate \eqref{dis-est} fails,
however we also want to prove  inhomogeneous Strichartz estimates for some non-admissible pairs. More precisely, we obtain the following result on the inhomogeneous Strichartz estimate.

 \begin{theorem}[Inhomogeneous Strichartz estimate]\label{thm-in} Let $\mathcal{L}_V=-\Delta+V(x)$  be given as above.
Then the inhomogeneous Strichartz estimate  holds for a constant $C$ and $s\in A_{\nu_0}$
\begin{equation}\label{Str-est-in'}
\left\|\int_0^t  e^{i(t-\sigma)\mathcal{L}_V}F(\sigma) d\sigma\right\|_{L^2(\R;L^{\frac{2n}{n-2s},2})}\leq
C\|F\|_{L^2(\R;L^{\frac{2n}{n+2(2-s)},2})}
\end{equation}
where
 \begin{equation}\label{Av}
 A_{\nu_0}=\left[\frac{n}{2(n-1)},\frac{3n-4}{2(n-1)}\right]\cap R_{\nu_0}.
\end{equation}
\end{theorem}

\vspace{0.2cm}

\begin{remark}The set $A_{\nu_0}$ is an intersection of two sets, the first set is related to the known result of the inhomogeneous Strichartz estimates in \cite{Foschi,Koh,Vilela,Schippa}  when $V=0$ and
the second set $R_{\nu_0}$ is from Theorem \ref{thm:w-resolvent}.
The picture of inhomogeneous Strichartz estimate is far to be completed even in the case without potential.
\end{remark}

\vspace{0.2cm}

Finally we introduce some notations. We use $A\lesssim B$ to denote
$A\leq CB$ for some large constant C which may vary from line to
line and depend on various parameters, and similarly we use $A\ll B$
to denote $A\leq C^{-1} B$. We employ $A\sim B$ when $A\lesssim
B\lesssim A$. If the constant $C$ depends on a special parameter
other than the above, we shall denote it explicitly by subscripts.
For instance, $C_\epsilon$ should be understood as a positive
constant not only depending on $p, q, n$, and $M$, but also on
$\epsilon$. Throughout this paper, pairs of conjugate indices are
written as $p, p'$, where $\frac{1}p+\frac1{p'}=1$ with $1\leq
p\leq\infty$. \vspace{0.2cm}

{\bf Acknowledgments:}\quad  J. Zhang and J. Zheng were supported by  NSFC Grants (11771041, 11831004, 11901041) and H2020-MSCA-IF-2017(790623). H. Mizutani is partially supported by JSPS KAKENHI Grant Number JP17K14218. We are grateful to the anonymous referee  for helpful comments.\vspace{0.2cm}

\section{The proof of the weighted resolvent estimate}

In this section, we prove  the uniform weighted resolvent estimates which are the key point to prove the other two theorems.

\begin{proof}[{\bf The proof of  Theorem \ref{thm:w-resolvent}}]
To prove Theorem \ref{thm:w-resolvent} although we follow the idea in \cite{BPST}, some modifications and improvements are required due to the reason
 that we have to replace the multiplier $r e^{-2r\tau}\phi(r)\partial_r \bar{v}$ by $r^\beta e^{-2r\tau}\phi(r)\partial_r \bar{v}$ which brings much harder treating terms
in  the weighted Hardy's inequality. By the duality, we only need to prove \eqref{in-local-sm1} with  $R_{\nu_0}\ni \alpha\geq1$. Indeed, if we could prove
\begin{equation}
\|r^{-\alpha}(\LL_V-\sigma)^{-1}r^{-2+\alpha}\|_{L^{2}(\R^n)\to L^2(\R^n)}\leq C, \quad 1\leq \alpha<\alpha_0
\end{equation}
where $\alpha_0=3/2$ or $1+\frac{\nu_0^2}{1-2\nu_0^2}$, by taking the adjoint of this estimate and replacing $\sigma$ by $\bar{\sigma}$,
we also have
\begin{equation}
\|r^{-2+\alpha}(\LL_V-\sigma)^{-1}r^{-\alpha}\|_{L^{2}(\R^n)\to L^2(\R^n)}\leq C, \quad 1\leq \alpha<\alpha_0
\end{equation}
which shows
\begin{equation}
\|r^{-\alpha'}(\LL_V-\sigma)^{-1}r^{-2+\alpha'}\|_{L^{2}(\R^n)\to L^2(\R^n)}\leq C, \quad 2-\alpha_0\leq \alpha'\leq1,
\end{equation}
where $\alpha'=2-\alpha$. So we only need to prove \eqref{in-local-sm1} with $1\leq \alpha<\alpha_0$.\vspace{0.2cm}

Let $z=\sqrt{-\sigma}$ with the branch such that $\mathrm{Re} z=\tau>0$. Then given $f\in L^2(\R^n)$ and $\sigma\in \C\setminus \R^+$, consider
the Helmholtz equation
\begin{equation}\label{helm-eq}
\LL_Vu+z^2u=f.
\end{equation} By density argument, we can take $f\in \mathcal{C}_0^\infty(\R^n)$. Then $u$ is a classical solution of  \eqref{helm-eq} and define
 $v(r,\theta): (0,\infty)\times \mathbb{S}^{n-1}\to \C$ by
 $$v(r,\theta)=r^{\frac{n-1}2} e^{rz}u(r,\theta).$$
  Then we see that
 \begin{equation*}
 \begin{split}
 \partial_r v=&r^{\frac{n-1}2} e^{rz}\left(\frac{n-1}{2r}u+zu+\partial_r u\right),\\
-\partial_r^2 v=&r^{\frac{n-1}2} e^{rz}\left(-\partial_r^2 u-2\Big(\frac{n-1}{2r}+z\Big)\partial_r u-\Big(\frac{(n-1)(n-3)}{4r^2}+\frac{(n-1)z}{r}+z^2 \Big)u\right),\\
z\partial_r v=&r^{\frac{n-1}2} e^{rz}\Big( z\partial_r u+\big(\frac{(n-1)z}{2r}+z^2 \big)u\Big).
\end{split}
\end{equation*}
Therefore, $v$ satisfies
 \begin{equation}\label{v}
 \begin{split}
&-\partial_r^2 v+2z\partial_r v+\left(\frac{(n-1)(n-3)}{4}-\Delta_\theta+V_0(\theta)\right)\frac{v}{r^2}
\\&=r^{\frac{n-1}2} e^{rz}\left(-\partial_r^2 u-\frac{n-1}{r}\partial_r u+\Big(-\Delta_\theta+V_0(\theta)+z^2 \Big)u\right)
\\&=r^{\frac{n-1}2} e^{rz}f.
\end{split}
\end{equation}
For fixed $M>m>0$, let $\phi=\phi_{m,M}(r)$ be a smooth cut-off function such that $0\leq\phi\leq1$ with being zero outside $[0, M+1]$ and equaling to $1$ on $[m,M]$.  By multiplying \eqref{v}
by $r^\beta e^{-2r\tau}\phi(r)\partial_r \bar{v}$ with $\beta$ being chosen later and taking the real part, we show that
 \begin{equation*}
 \begin{split}
&-\frac12 r^{\beta}e^{-2r\tau}\phi(r)\partial_r|\partial_r v|^2+2\tau r^{\beta} e^{-2r\tau}\phi(r)|\partial_r v|^2\\&+\frac{1}{2r^{2-\beta}}e^{-2r\tau}\phi(r) \left(\frac{(n-1)(n-3)}{4}+V_0(\theta)\right)\partial_r|v|^2
+\frac{1}{r^{2-\beta}}e^{-2r\tau}\phi(r) \mathrm{Re}(-\Delta_\theta v\partial_r\bar{v})\\& =r^{\frac{n-1}2+\beta}\phi(r) \mathrm{Re}\left(e^{r(z-2\tau)}\partial_r \bar{v} f\right).
\end{split}
\end{equation*}
Integrating the above formula on $(0,\infty)\times \mathbb{S}^{n-1}$ but with volume $drd\theta$ and performing the integration by parts, we have
 \begin{equation*}
 \begin{split}
&\frac12 \int_0^\infty\int_{\mathbb{S}^{n-1}} \partial_r\left(r^{\beta}e^{-2r\tau}\phi(r)\right)|\partial_r v|^2 dr d\theta
\\&+2\tau \int_0^\infty\int_{\mathbb{S}^{n-1}} r^{\beta}e^{-2r\tau}\phi(r)|\partial_r v|^2 dr d\theta\\&-\frac{1}{2}\int_0^\infty\int_{\mathbb{S}^{n-1}}\partial_r\left(r^{-2+\beta} e^{-2r\tau}\phi(r)\right) \left(\frac{(n-1)(n-3)}{4}+V_0(\theta)\right)|v|^2 dr d\theta
\\ \quad &-\frac{1}{2}\int_0^\infty\int_{\mathbb{S}^{n-1}}\partial_r\left(r^{-2+\beta} e^{-2r\tau}\phi(r)\right)|\nabla_\theta v|^2dr d\theta
\\&=\int_0^\infty\int_{\mathbb{S}^{n-1}} r^{\frac{n-1}2+\beta}\phi(r) \mathrm{Re}\left(e^{r(z-2\tau)}\partial_r \bar{v} f\right) dr d\theta.
\end{split}
\end{equation*}
Furthermore we compute that
 \begin{equation*}
 \begin{split}
&\frac12 \int_0^\infty\int_{\mathbb{S}^{n-1}} e^{-2r\tau}\phi(r)r^{\beta-1} \left(\beta- 2r\tau \right)|\partial_r v|^2 dr d\theta+\frac12 \int_0^\infty\int_{\mathbb{S}^{n-1}} r^{\beta}e^{-2r\tau}\phi'(r)|\partial_r v|^2 dr d\theta
\\&+2\tau \int_0^\infty\int_{\mathbb{S}^{n-1}} r^{\beta}e^{-2r\tau}\phi(r)|\partial_r v|^2 dr d\theta\\&+\frac{1}{2}\int_0^\infty\int_{\mathbb{S}^{n-1}}e^{-2r\tau}\phi(r) r^{-3+\beta}\left((2-\beta)+2r\tau \right) \left(\frac{(n-2)^2}{4}-\frac14+V_0(\theta)\right)|v|^2 dr d\theta
\\&-\frac12 \int_0^\infty\int_{\mathbb{S}^{n-1}} r^{-2+\beta}e^{-2r\tau}\phi'(r)\left(\frac{(n-2)^2}{4}-\frac14+V_0(\theta)\right)|v|^2 dr d\theta
\\ \quad &+\frac{1}{2}\int_0^\infty\int_{\mathbb{S}^{n-1}}e^{-2r\tau}\phi(r) r^{-3+\beta}\left((2-\beta)+2r\tau \right)  |\nabla_\theta v|^2 dr d\theta
\\&-\frac12 \int_0^\infty\int_{\mathbb{S}^{n-1}} r^{-2+\beta}e^{-2r\tau}\phi'(r)|\nabla_\theta v|^2 dr d\theta
\\&=\int_0^\infty\int_{\mathbb{S}^{n-1}} r^{\frac{n-1}2+\beta}\phi(r) \mathrm{Re}\left(e^{r(z-2\tau)}\partial_r \bar{v} f\right) dr d\theta.
\end{split}
\end{equation*}
Therefore we show
 \begin{equation}
 \begin{split}
&\frac12 \int_0^\infty\int_{\mathbb{S}^{n-1}} e^{-2r\tau}\phi(r) r^{\beta-1}\left(\left(\beta+ 2r\tau \right)|\partial_r v|^2-\left((2-\beta)+2r\tau \right) \frac{|v|^2}{4r^2}\right)  dr d\theta
\\&+\frac{1}{2}\int_0^\infty\int_{\mathbb{S}^{n-1}}e^{-2r\tau}\phi(r) r^{-3+\beta}\left((2-\beta)+2r\tau \right)  \left(\left(\frac{(n-2)^2}{4}+V_0(\theta)\right)|v|^2+|\nabla_\theta v|^2\right) dr d\theta
\\&+\frac12 \int_0^\infty\int_{\mathbb{S}^{n-1}} r^{\beta}e^{-2r\tau}\phi'(r)\left(|\partial_r v|^2+\frac{1}{4r^2}|v|^2-\frac1{r^2}\big(|\nabla_\theta v|^2+(V_0(\theta)+\frac{(n-2)^2}4)|v|^2\big)\right) dr d\theta
\\&=\int_0^\infty\int_{\mathbb{S}^{n-1}} r^{\frac{n-1}2+\beta}\phi(r) \mathrm{Re}\left(e^{r(z-2\tau)}\partial_r \bar{v} f\right) dr d\theta.
\end{split}
\end{equation}
On the other hand,  $-\Delta_\theta+V_0(\theta)+(n-2)^2/4$ is positive on ${\mathbb{S}^{n-1}}$ with the smallest eigenvalue $\nu_0^2>0$, that is,
\begin{equation}\label{lowb}
\int_{\mathbb{S}^{n-1}}\left(\left(\frac{(n-2)^2}{4}+V_0(\theta)\right)|v|^2+|\nabla_\theta v|^2\right) d\theta\geq \nu_0^2\int_{\mathbb{S}^{n-1}} |v(r,y)|^2 d\theta\geq 0.
\end{equation}
Hence we show for $\forall \epsilon>0$
 \begin{equation}\label{est:v}
 \begin{split}
&\frac12 \int_0^\infty\int_{\mathbb{S}^{n-1}} e^{-2r\tau}\phi(r) r^{\beta-1}\left(\left(\beta+ 2r\tau \right)|\partial_r v|^2+\left((2-\beta)+2r\tau \right)(\nu_0^2-\frac14) \frac{|v|^2}{r^2}\right)  dr d\theta
\\&+\frac12 \int_0^\infty\int_{\mathbb{S}^{n-1}} r^\beta e^{-2r\tau}\phi'(r)\left(|\partial_r v|^2+\frac{1}{4r^2}|v|^2-\frac1{r^2}\big(|\nabla_\theta v|^2+(V_0(\theta)+\frac{(n-2)^2}4)|v|^2\big)\right) dr d\theta
\\&\leq\int_0^\infty\int_{\mathbb{S}^{n-1}} r^{\frac{n-1}2+\beta}\phi(r) \mathrm{Re}\left(e^{r(z-2\tau)}\partial_r \bar{v} f\right) dr d\theta
\\&\leq \frac{1}{4\epsilon^2}\|r^{\frac{1+\beta}2}f\|_{L^2}^2+\epsilon^2\int_0^\infty\int_{\mathbb{S}^{n-1}} \phi(r) e^{-2r\tau}r^{\beta-1}|\partial_r v|^2  dr d\theta.
\end{split}
\end{equation}

For our purpose, we first need the following lemma.

\begin{lemma}\label{to0}Let $0\leq \beta\leq 1$, we have following estimate for $ m\to0, M\to \infty$
\begin{equation}
\int_0^\infty\int_{\mathbb{S}^{n-1}} r^\beta e^{-2r\tau}\phi'(r)\left(|\partial_r v|^2+\frac{1}{4r^2}|v|^2-\frac1{r^2}\big(|\nabla_\theta v|^2+(V_0(\theta)+\frac{(n-2)^2}4)|v|^2\big)\right) dr d\theta\geq 0.
\end{equation}
\end{lemma}
We postpone the proof in the next subsection.\vspace{0.2cm}

By  taking the limits $m\to 0$ and $M\to \infty$ and using Lemma \ref{to0} and \eqref{est:v}, we have
 \begin{equation*}
 \begin{split}
\frac12 \int_0^\infty\int_{\mathbb{S}^{n-1}} e^{-2r\tau} &r^{\beta-1}\Big(\left(\beta+ 2r\tau-2\epsilon^2 \right)|\partial_r v|^2
\\&+\big((2-\beta)+2r\tau \big)(\nu_0^2-\frac14) \frac{|v|^2}{r^2}\Big)  dr d\theta
\leq \frac{1}{4\epsilon^2}\|r^{\frac{1+\beta}2}f\|_{L^2}^2.
\end{split}
\end{equation*}
Furthermore we obtain for $0<\beta\leq1$
\begin{equation}\label{est:reduction}
 \begin{split}
(1&-\frac{2\epsilon^2}\beta) \int_0^\infty\int_{\mathbb{S}^{n-1}} e^{-2r\tau} r^{\beta-1}\left(\beta+ 2r\tau\right)|\partial_r v|^2dr d\theta
\\&+(\nu_0^2-\frac14)\int_0^\infty\int_{\mathbb{S}^{n-1}} e^{-2r\tau} r^{\beta-1}\big((2-\beta)+2r\tau \big) \frac{|v|^2}{r^2}  dr d\theta
\leq \frac{1}{2\epsilon^2}\|r^{\frac{1+\beta}2}f\|_{L^2}^2.
\end{split}
\end{equation}

{\bf Case 1: $\nu_0>1/2$.} Since $0<\beta\leq 1$ and $r\tau>0$, we have
\begin{equation}
 \begin{split}
\int_0^\infty\int_{\mathbb{S}^{n-1}} e^{-2r\tau} r^{\beta-1} \frac{|v|^2}{r^2}  dr d\theta
\leq C_{\nu_0}\|r^{\frac{1+\beta}2}f\|_{L^2}^2,
\end{split}
\end{equation}
which implies $$\|r^{-\alpha} u\|_{L^2(\R^n)}\leq C_{\nu_0}\|r^{2-\alpha}f\|_{L^2(\R^n)}, \quad \beta=3-2\alpha.$$
Since $0<\beta\leq1$, we have showed that if $\nu_0>1/2$ and $1\leq \alpha<3/2$.
\begin{equation}
\|r^{-\alpha}(\LL_V-z^2)^{-1}r^{\alpha-2}\|_{L^2\to L^2}\leq C.
\end{equation}\vspace{0.2cm}

{\bf Case 2: $0<\nu_0\leq 1/2$.} In this case, we need a weighted Hardy's inequality

\begin{lemma}[Weighted Hardy's inequality]\label{lem:w-hardy} Let $w\in \mathcal{C}^2(\R^+\setminus\{0\};\R)$ satisfy
\begin{equation}\label{equ:wcond}
w(r)\geq0, w'(r)\leq 0, \quad r(w'(r)^2+2w(r)w''(r))\geq 2w(r)w'(r), \forall r\geq0.
\end{equation}
Let $g:\R^+\to\C$ be such that
\begin{equation}\label{equ:funccond}
  \int_0^\infty \big(w^2(r)|g'|^2+(w'(r))^2|g|^2\big)\;dr<+\infty
\end{equation}
and
\begin{equation}\label{equ:fcond}
  \liminf_{r\to0}w(r)w'(r)|g(r)|^2=0.
\end{equation} Then
\begin{equation}
\int_0^\infty w^2\frac{|g(r)|^2}{r^2}dr\leq 4\int_0^\infty w^2|g'(r)|^2dr.
\end{equation}
\end{lemma}

Next we use the modified weighted Hardy's inequality to show

\begin{lemma}\label{lem:asmalllv}
Let $\max\big\{0,1-2\nu_0\big\}<\beta\leq1$, then we have
\begin{align}\label{equ:asmalllv}
 & \int_0^\infty\int_{\mathbb{S}^{n-1}} e^{-2r\tau} r^{\beta-1}\left(\beta+ 2r\tau\right)|\partial_r v|^2dr d\theta \\\nonumber
 \geq&\frac14 \int_0^\infty\int_{\mathbb{S}^{n-1}} e^{-2r\tau} r^{\beta-1}\big(\beta+2r\tau \big) \frac{|v|^2}{r^2}  dr d\theta.
\end{align}
\end{lemma}
We postpone the proof of the two lemmas at the end of this section.

Using Lemma \ref{lem:asmalllv} and  \eqref{est:reduction}, we obtain
\begin{equation}\label{est:reduction'}
 \begin{split}
&C\frac{1}{2\epsilon^2}\|r^{\frac{1+\beta}2}f\|_{L^2}^2\\\geq& \frac14(1-\frac{2\epsilon^2}\beta) \int_0^\infty\int_{\mathbb{S}^{n-1}} e^{-2r\tau} r^{\beta-1}\left(\beta+ 2r\tau\right)\frac{|v|^2}{r^2} dr d\theta
\\&+(\nu_0^2-\frac14)\int_0^\infty\int_{\mathbb{S}^{n-1}} e^{-2r\tau} r^{\beta-1}\big((2-\beta)+2r\tau \big) \frac{|v|^2}{r^2}  dr d\theta\\
\geq& \frac\beta 4 (1-\frac{2\epsilon^2}\beta) \int_0^\infty\int_{\mathbb{S}^{n-1}} e^{-2r\tau} r^{\beta-1}\frac{|v(r)|^2}{r^2}\;d\theta\;dr
\\&+(2-\beta)(\nu_0^2-\frac14)\int_0^\infty\int_{\mathbb{S}^{n-1}} e^{-2r\tau} r^{\beta-1} \frac{|v|^2}{r^2}  dr d\theta,
\end{split}
\end{equation}
which implies $$\|r^{-\alpha} u\|_{L^2(\R^n)}\leq C_{\nu_0}\|r^{2-\alpha}f\|_{L^2(\R^n)}, \quad \beta=3-2\alpha$$
provided that
 $$\frac\beta 4>(2-\beta)(\frac14-\nu_0^2)\Leftrightarrow \beta>1-\frac{2\nu_0^2}{1-2\nu_0^2}.$$
 Thus we have shown
\begin{equation}
\|r^{-\alpha}(\LL_V-z^2)^{-1}r^{\alpha-2}\|_{L^2\to L^2}\leq C,\quad 1\leq \alpha<1+\frac{\nu_0^2}{1-2\nu_0^2}.
\end{equation}
Therefore we conclude the proof of  Theorem \ref{thm:w-resolvent} if we could prove Lemma \ref{to0}, Lemma \ref{lem:w-hardy}  and Lemma \ref{lem:asmalllv}.
\end{proof}
\vspace{0.2cm}

To complete the proof, we have to prove the lemmas which will be done in the rest of this section. In the proof, we have to be careful the factor associated with the index $\beta$.

\subsection{ Proof of Lemma \ref{to0}}

Before proving Lemma \ref{to0}, we show the following lemmas.
\begin{lemma}\label{lem:uhelmequlv} Let $\sigma\notin\R^+$ and $f\in \mathcal{C}_0^\infty(\R^n)$. Assume that $u$ is the classical solution to
\begin{equation}\label{helm-eqinblv}
\LL_V u-\sigma u=f.
\end{equation}
Then, $u\in \dot{H}^1(\R^n)$ if $\sigma=0$ and $u\in {H}^1(\R^n)$ if $\sigma\neq 0$, and
\begin{equation}\label{equ:nabluestlv}
\int_{\R^n}\frac{|\partial_r u|^2}{|x|^{1-\beta}}\;dx+\int_{\R^n}\frac{|u|^2}{|x|^{3-\beta}}\;dx<+\infty.
\end{equation}
Here $\beta>\max\{0,1-2\nu_0\}$
where $\nu_0$ is the positive square root of the smallest eigenvalue of the operator $-\Delta_\theta+V_0(\theta)+(n-2)^2/4$.

\end{lemma}

To prove Lemma \ref{lem:uhelmequlv}, we first show the modified Hardy inequality.
\begin{lemma}\label{lem:w-hardy'} Let $k\neq\frac{n}{2}$. There holds
\begin{equation}\label{equ:modhardylv}
 \int_{0}^\infty\int_{\mathbb{S}^{n-1}}\frac{|f|^2}{r^{2k}}\;r^{n-1}drd\theta\leq\frac{4}{(n-2k)^2} \int_{0}^\infty\int_{\mathbb{S}^{n-1}} \frac{|\partial_r f|^2}{r^{2k-2}}\;r^{n-1}drd\theta.
\end{equation}
\end{lemma}

\begin{proof}
First, by the sharp Hardy's inequality \cite{KSWW}, we have
  \begin{align}\label{equ:byharslv}
    \int_{0}^\infty \frac{|f(r,\theta)|^2}{r^{2k}}\;r^{n-1}dr\leq&\frac{4}{(n-2)^2}\int_{0}^\infty \Big|\partial_r\Big(\frac{f(r,\theta)}{r^{k-1}}\Big)\Big|^2\;r^{n-1}dr\\\nonumber
    =&\frac{4}{(n-2)^2}\int_{0}^\infty \frac{1}{r^{2k-2}}\big|\partial_r f-(k-1)f(r,\theta)r^{-1}\big|^2 r^{n-1}\;dr.
  \end{align}
Noting that
$$\big|\partial_r f-(k-1)r^{-1}f\big|^2=|\partial_r f|^2-2(k-1)r^{-1}\cdot{\rm Re}\big(\bar{f}\partial_r f\big)+(k-1)^2\tfrac{|f|^2}{r^2},$$
we get
\begin{align*}
  &\int_{0}^\infty \frac{1}{r^{2k-2}}\big|\partial_r f-(k-1)f(r,\theta)r^{-1}\big|^2 r^{n-1}\;dr\\
  =&\int_{0}^\infty \left(|\partial_r f|^2-2(k-1)r^{-1}\cdot{\rm Re}\big(\bar{f}\partial_r f\big)+(k-1)^2\tfrac{|f|^2}{r^2}\right)\;r^{n-1}dr\\
  =&\int_{0}^\infty \frac{|\partial_r f|^2}{r^{2k-2}}\;r^{n-1}dr+(k-1)(n-k-1)\int_{0}^\infty \frac{|f|^2}{r^{2k}}\;r^{n-1}dr.
\end{align*}
  Plugging this into \eqref{equ:byharslv} yields
 \begin{equation}
 \begin{split}
  &\int_{0}^\infty \frac{|f(r,\theta)|^2}{r^{2k}}\;r^{n-1}dr \\&\leq\frac{1}{(n-2)^2/4-(k-1)(n-k-1)}\int_{0}^\infty \frac{|\partial_r f|^2}{r^{2k-2}}\;r^{n-1}dr\\
  &=\frac{4}{(n-2k)^2}\int_{0}^\infty \frac{|\partial_r f|^2}{r^{2k-2}}\;r^{n-1}dr,
  \end{split}
  \end{equation}
 and so   \eqref{equ:modhardylv} follows. Therefore, we integrate on ${\mathbb{S}^{n-1}}$ to conclude the proof of Lemma \ref{lem:w-hardy'}.

\begin{proof}[{\bf Proof of Lemma \ref{lem:uhelmequlv}:}]
  We first consider the case $\sigma=0$, that is, $u$ solves
\begin{equation}\label{equ:sigma0lv}
  -\Delta u+r^{-2}V_0(\theta)u=f.
\end{equation}
Multiplying the above equality by $\bar{u}$ and integrating on $\R^n$, we obtain
$$\int_{\R^n}\Big(|\nabla u|^2+V_0(\theta)r^{-2}|u|^2\Big)\;dx={\rm Re}\int_{\R^n}f\bar{u}\;dx.$$
By Young's inequality and \cite[Proposition 1]{BPSS}, we have
\begin{align*}
 \|u\|_{\dot{H}^1}^2\sim\big\|\sqrt{\LL_V}u\big\|_{L^2}^2&=\int_{\R^n}\Big(|\nabla u|^2+V_0(\theta)r^{-2}|u|^2\Big)\;dx\\& \leq  C(\epsilon)\int_{\R^n} r^2 |f|^2\;dx+\epsilon \int_{\R^n}|\nabla u|^2\;dx.
\end{align*}
Hence,
$$\int_{\R^n}|\nabla u|^2\;dx\leq  C\int_{\R^n}r^2 |f|^2\;dx,$$
and so $u\in\dot{H}^1$.

Next, multiplying \eqref{equ:sigma0lv} by $\tfrac{\bar{u}}{r^{1-\beta}}$, integrating in $\R^n$ and taking real part, we obtain
\begin{align*}
\int_{0}^\infty \int_{\mathbb{S}^{n-1}} \frac{|\nabla u|^2}{r^{1-\beta}}\;r^{n-1}drd\theta-\frac12\int_{0}^\infty \int_{\mathbb{S}^{n-1}} |u|^2\Delta\big(r^{\beta-1}\big)\;r^{n-1}drd\theta\\ +\int_{0}^\infty \int_{\mathbb{S}^{n-1}}\frac{V_0(\theta)}{r^{3-\beta}}|u|^2\;r^{n-1}dr d\theta
={\rm Re}\int_{0}^\infty \int_{\mathbb{S}^{n-1}} \frac{f\bar{u}}{r^{1-\beta}}\;r^{n-1}drd\theta,
\end{align*}
which implies
\begin{align*}
&\int_{0}^\infty \int_{\mathbb{S}^{n-1}} \frac{|\nabla u|^2}{r^{1-\beta}}\;r^{n-1}drd\theta+\int_{0}^\infty \int_{\mathbb{S}^{n-1}} \big(V_0(\theta)+\tfrac{(1-\beta)(n-3+\beta)}2\big)\frac{|u|^2}{r^{3-\beta}}\;r^{n-1}drd\theta\\
&={\rm Re}\int_{0}^\infty \int_{\mathbb{S}^{n-1}}\frac{f\bar{u}}{r^{1-\beta}}\;r^{n-1}drd\theta.\end{align*}
Noting that $\nabla=(\partial_r, r^{-1}\nabla_\theta)$, \eqref{lowb} with $u=v$ implies
\begin{align*}
   & \int_{0}^\infty \int_{\mathbb{S}^{n-1}}  \frac{|\partial_r u|^2}{r^{1-\beta}}\;r^{n-1}drd\theta+\int_{0}^\infty \int_{\mathbb{S}^{n-1}} \big(\nu_0^2-\frac{(n-2)^2}4+\tfrac{(1-\beta)(n-3+\beta)}2\big)\frac{|u|^2}{r^{3-\beta}}\;r^{n-1}drd\theta \\
  \leq & C(\epsilon)\int_{0}^\infty \int_{\mathbb{S}^{n-1}}  |f|^2r^{1+\beta}\;r^{n-1}drd\theta+\epsilon \int_{0}^\infty \int_{\mathbb{S}^{n-1}} \frac{|u|^2}{r^{3-\beta}}\;r^{n-1}drd\theta .
\end{align*}
Using Lemma \ref{lem:w-hardy'} with $2k=3-\beta$, one has
$$\int_{0}^\infty \int_{\mathbb{S}^{n-1}} \frac{|u|^2}{r^{3-\beta}}\;r^{n-1}drd\theta\leq\frac{4}{(n-3+\beta)^2}\int_{0}^\infty \int_{\mathbb{S}^{n-1}} \frac{|\partial_r u|^2}{r^{1-\beta}}\;r^{n-1}drd\theta.$$
Hence, for
\begin{equation}\label{equ:acondbetalv}
  \nu_0^2-\frac{(n-2)^2}4+\tfrac{(1-\beta)(n-3+\beta)}2>-\tfrac{(n-3+\beta)^2}{4}, \Leftrightarrow \nu_0^2>\tfrac{(1-\beta)^2}4,
\end{equation} there holds
\begin{equation}\label{equ:anotsmalv}
\begin{split}
  &\int_{0}^\infty \int_{\mathbb{S}^{n-1}}  \frac{|\partial_r u|^2}{r^{1-\beta}}\;r^{n-1}drd\theta+\int_{0}^\infty \int_{\mathbb{S}^{n-1}} \frac{|u|^2}{r^{3-\beta}}\;r^{n-1}drd\theta \\&\leq C\int_{0}^\infty \int_{\mathbb{S}^{n-1}}  |f|^2r^{1+\beta}\;r^{n-1}drd\theta.
  \end{split}
\end{equation}

Next we consider the case $\sigma\neq0$.
Multiplying \eqref{helm-eqinblv} by $\bar{u}$ and integrating in $\R^n$, let $\sigma=\sigma_1+i\sigma_2$, we get
  \begin{align}\label{equ:reparteslv}
 \int_{\R^n}\Big(|\nabla u|^2+V_0(\theta)r^{-2}|u|^2\Big)\;dx-\sigma_1\int_{\R^n}|u|^2\;dx=&{\rm Re}\int_{\R^n}f\bar{u}\;dx,\\\label{equ:impartslv}
  -\sigma_2\int_{\R^n}|u|^2\;dx=&{\rm Im}\int_{\R^n}f\bar{u}\;dx.
  \end{align}

{\bf Case 1: $\sigma_2\neq0$.} It follows from \eqref{equ:impartslv} that
$$|\sigma_2|\int_{\R^n}|u|^2\;dx\leq\frac2{|\sigma_2|}\int_{\R^n}|f|^2\;dx+\frac{|\sigma_2|}4\int_{\R^n}|u|^2\;dx\Rightarrow~u\in L^2(\R^n).$$
Combining this with \eqref{equ:reparteslv}, we obtain
  \begin{align}\label{equ:reparteslv'}
 \|\nabla u\|_{L^2(\R^n)}^2\sim \|\sqrt{\LL_V}u\|_{L^2}^2&=\int_{\R^n}\Big(|\nabla u|^2+V_0(\theta)r^{-2}|u|^2\Big)\;dx
 \\&=\sigma_1\int_{\R^n} |u|^2\;dx-{\rm Re}\int_{\R^n} f\bar{u}\; dx<\infty  \end{align}
Hence $u\in\dot{H}^1.$
Multiplying \eqref{helm-eqinblv} by $\frac{\bar{u}}{r^{1-\beta}}$ and integrating in $\R^n$, we get
\begin{align}\label{equ:reaa1lv}
&\int_{0}^\infty \int_{\mathbb{S}^{n-1}} \frac{|\nabla u|^2}{r^{1-\beta}}\;r^{n-1}drd\theta+\int_{0}^\infty \int_{\mathbb{S}^{n-1}} \big(V_0(\theta)+\tfrac{(1-\beta)(n-3+\beta)}2\big)\frac{|u|^2}{r^{3-\beta}}\;r^{n-1}drd\theta\\ \nonumber
&-\sigma_1\int_{0}^\infty\int_{\mathbb{S}^{n-1}} \frac{|u|^2}{r^{1-\beta}}\;r^{n-1}drd\theta ={\rm Re}\int_{0}^\infty \int_{\mathbb{S}^{n-1}}\frac{f\bar{u}}{r^{1-\beta}}\;r^{n-1}drd\theta.\end{align}
  and
  \begin{equation}\label{equ:imaa2lv}
    -\sigma_2\int_{0}^\infty\int_{\mathbb{S}^{n-1}} \frac{|u|^2}{r^{1-\beta}}\;r^{n-1}drd\theta={\rm Im}\int_{0}^\infty\int_{\mathbb{S}^{n-1}}\ \frac{f\bar{u}}{r^{1-\beta}}\;r^{n-1}drd\theta.
  \end{equation}
This together with Young's inequality yields $\int_{\R^n}\frac{|u|^2}{r^{1-\beta}}\;dx<+\infty.$ Applying this fact to \eqref{equ:reaa1lv}, and by the same argument as \eqref{equ:anotsmalv}, we obtain that
 if $\nu_0^2>\tfrac{(1-\beta)^2}4,$ there holds
\begin{equation*}
\begin{split}
  &\int_{0}^\infty \int_{\mathbb{S}^{n-1}}  \frac{|\partial_r u|^2}{r^{1-\beta}}\;r^{n-1}drd\theta+\int_{0}^\infty \int_{\mathbb{S}^{n-1}} \frac{|u|^2}{r^{3-\beta}}\;r^{n-1}drd\theta \\&\leq C\int_{0}^\infty \int_{\mathbb{S}^{n-1}}  |f|^2r^{1+\beta}\;r^{n-1}drd\theta.
  \end{split}
\end{equation*}

{\bf Case 2: $\sigma_2=0$.} In this case, we have $\sigma_1<0$ due to $\sigma\notin \R^+$. Using  \eqref{equ:reparteslv}, we obtain $u\in{H}^1.$
So \eqref{equ:nabluestlv} follows from \eqref{equ:reaa1lv}.
\end{proof}

With Lemma \ref{lem:uhelmequlv} in hand, we now prove Lemma \ref{to0}. Note the fact that the support of $\phi'(r)$ is compact and belongs to $[0,m]\cup[M, M+1]$. One has $0\leq \phi'\leq C/m$ on $[0,m]$ and $-C\leq \phi'\leq 0$ on $[M,M+1]$.
Thus by \eqref{lowb} it suffices to show the negative terms
\begin{equation}\label{M0}
\int_M^{M+1}\int_{\mathbb{S}^{n-1}} r^\beta e^{-2r\tau}\left(|\partial_r v|^2+\frac{1}{4r^2}|v|^2\right) dr d\theta\to 0, \quad \text{as}~ M\to \infty;
\end{equation}
and
\begin{equation}\label{m0}
\int_0^{m}\int_{\mathbb{S}^{n-1}} r^\beta e^{-2r\tau}\frac1{r^2}\left(|\nabla_\theta v|^2+(V_0(\theta)+\frac{(n-2)^2}4)|v|^2\right) dr d\theta\to 0, \quad \text{as} ~m\to 0.
\end{equation}
It is  enough to show that there exists a
sequence $M_n\to\infty$ along which \eqref{M0} holds. We note that
$$\partial_rv=r^\frac{n-1}{2}e^{rz}\big(\partial_ru+zu+\tfrac{n-1}{2r}u\big).$$
By using the modified Hardy inequality \eqref{equ:modhardylv}, we have 
\begin{align*}
&\int_M^{M+1}\int_{\mathbb{S}^{n-1}} r^\beta e^{-2r\tau}\left(|\partial_r v|^2+\frac{1}{4r^2}|v|^2\right) dr d\theta\\
\leq&C\int_{M}^{M+1}\int_{{\mathbb{S}^{n-1}}}r^{\beta} \big(|\partial_ru|^2+|z|^2|u|^2\big)\;d\theta\; r^{n-1}\;dr\\
\leq&C(n,|z|)\int_{M}^{M+1} r^\beta g(r)\;dr,
\end{align*}
where
$$g(r)=r^{n-1}\int_{{\mathbb{S}^{n-1}}}\big(|\partial_ru|^2+|u|^2\big)\;d\theta.$$
By Lemma \ref{lem:uhelmequlv}, we get
$$\int_0^\infty g(r)\;dr<+\infty.$$
It thus follows that, given $\mu_j>0$, there exists a sequence $M_n^{(j)}\to\infty$ such that
$$\int_{M_n^{(j)}}^{M_n^{(j)}+1}g(r)\;dr<\frac{\mu_j}{M_n^{(j)}},$$
because otherwise the integral $\int_0^\infty g(r)\;dr$ would diverge. Using a diagonal argument
it thus follows that there exists a sequence $M_n\to\infty$ such that for $\beta\leq1$ (i.e. $\alpha\geq1$)
$$\int_{M_n}^{M_n+1}r^\beta g(r)\;dr\to0\quad \text{as}\quad n\to\infty,$$
which implies \eqref{M0} along a sequence.

On the other hand, using Lemma \ref{lem:uhelmequlv}, we have for $\beta\geq0$
\begin{align*}
&\int_0^{m}\int_{{\mathbb{S}^{n-1}}}  r^{\beta-2}e^{-2r\tau}\left(|\nabla_{\theta}  v|^2+|v|^2\right) dr d\theta\\
\leq&\int_0^{m}\int_{{\mathbb{S}^{n-1}}}  r^{-2}\left(|\nabla_{\theta}  u|^2+|u|^2\right)\;d\theta\;r^{n-1} \;dr \\
\lesssim&\int_{r\leq m}\left(|\nabla u|^2+\frac{|u|^2}{r^{2}}\right)\;dx\\
\to&0\quad\text{as}\quad m\to0.
\end{align*}

\end{proof}

\subsection{Proof of Lemmas \ref{lem:w-hardy} and  \ref{lem:asmalllv}}

\begin{proof}[{\bf The proof of Lemma \ref{lem:w-hardy}}]
This is a modification of the weighted Hardy inequality in \cite[Lemma 2.2]{BPST}. We just modify the argument in \cite{BPST} to prove it.
Let the operator $G$ be defined as
$$G:=\frac1i\big(w\partial_r+\tfrac12w'\big).$$
It follows from \eqref{equ:fcond} that there exists a sequence $\{r_j\}_j:~r_j\to0$ such that
$$\lim_{j\to\infty}w(r_j)w'(r_j)|g(r_j)|^2=0.$$
This together with \eqref{equ:funccond} and \eqref{equ:wcond} with $ww'\leq 0$ yields that
\begin{align*}
  \|Gg\|_{L^2(\R^+)}^2= &\int_0^\infty \big(w^2|g'|^2+\tfrac14(w')^2|g|^2+\tfrac12ww'\partial_r|g|^2\big)\;dr \\
  = & \lim_{j\to\infty}\bigg\{\int_{r_j}^\infty\big(w^2|g'|^2+\tfrac14(w')^2|g|^2-\tfrac12\partial_r(ww')|g|^2\big)\;dr
  +\tfrac12w(r)w'(r)|g(r)|^2\Big|_{r=r_j}^{r=+\infty}\bigg\}\\
  \leq&\int_0^\infty\Big(w^2|g'|^2-\big(\tfrac14(w')^2+\tfrac12ww''\Big)|g|^2\big)\;dr.
\end{align*}
On the other hand,  for the function $m(r)=-\frac{w(r)}{2r}$, a simple computation shows that
\begin{align}\nonumber
  0\leq & \big\|(G-im)g\big\|_{L^2(\R^+)}^2 \\\nonumber
  = & \|Gg\|_{L^2(\R^+)}^2-\big\langle (w m'-m^2)g,g\big\rangle\\\nonumber
  \leq&\int_0^\infty\big(w^2|g'|^2-\big(\tfrac14(w')^2|g|^2+\tfrac12ww''\big)|g|^2\big)\;dr
  -\big\langle (w m'-m^2)f,f\big\rangle\\\label{equ:fm}
  =&\int_0^\infty\Big(w^2|g'|^2-\big(\tfrac14(w')^2+\tfrac12ww''+w m'-m^2\big)|g|^2\Big)\;dr.
\end{align}
Noting that $m=-\frac{w}{2r}$, we have by \eqref{equ:wcond}
$$\tfrac14(w')^2+\tfrac12ww''+w m'-m^2=\tfrac14(w')^2+\tfrac12ww''-\tfrac{ww'}{2r}+\tfrac{w^2}{4r^2}
\geq\tfrac{w^2}{4r^2}.$$
Plugging this into \eqref{equ:fm}, we obtain
$$\int_0^\infty w^2\frac{|g|^2}{r^2}\;dr\leq 4\int_0^\infty w^2|g'(r)|^2\;dr.$$

\end{proof}

\begin{proof}[{\bf The proof of Lemma \ref{lem:asmalllv}:}]  Let
\begin{equation}\label{equ:wdef}
  w(r)=e^{-r\tau}r^{\frac{\beta-1}2}(\beta+2\tau r)^{1/2}.
\end{equation}
We first verify the assumption \eqref{equ:wcond} on $w(r)$ when $0<\beta\leq1$. A simple computation shows
\begin{equation}\label{equ:deri}
  w'(r)=e^{-r\tau}r^{(\beta-1)/2}(\beta+2\tau r)^{-1/2}[-2\tau^2r+\tfrac12(\beta-1)\beta r^{-1} ]\leq0
\end{equation}
for $0<\beta\leq1.$ Secondly we have
\begin{align*}
  w''(r)= & e^{-r\tau}r^{(\beta-1)/2}(\beta+2\tau r)^{-3/2}[-2\tau^2r+\tfrac12(\beta-1)\beta r^{-1} ]^2\\
+&e^{-r\tau}r^{(\beta-1)/2} (\beta+2\tau r)^{-3/2}[-2\tau(\beta-1)\beta r^{-1} -2\tau^2\beta-\tfrac12(\beta-1)\beta^2 r^{-2}]
\end{align*}
which implies
\begin{align*}
 &w'(r)^2+2w(r)w''(r)\\=&e^{-2r\tau}r^{\beta-1}(\beta+2\tau r)^{-1}[-2\tau^2r+\frac12(\beta-1)\beta r^{-1} ]^2\\
&+e^{-2r\tau}r^{\beta-1}(\beta+2\tau r)^{-3/2}[-2\tau^2r+\frac12(\beta-1)\beta r^{-1} ]^2\\
&+e^{-2r\tau}r^{\beta-1} (\beta+2\tau r)^{-3/2}[-2\tau(\beta-1)\beta r^{-1} -2\tau^2\beta-\frac12(\beta-1)\beta^2 r^{-2}]\\=&e^{-2r\tau}r^{\beta-1}(\beta+2\tau r)^{-1}[-2\tau^2r+\frac12(\beta-1)\beta r^{-1} ]^2\\
&+e^{-2r\tau}r^{\beta-1}(\beta+2\tau r)^{-3/2}[-2\tau^2r+\frac12(\beta-1)\beta r^{-1} ]^2\\
&+e^{-2r\tau}r^{\beta-1} (\beta+2\tau r)^{-3/2}[-2\tau(\beta-1)\beta r^{-1} -2\tau^2\beta-\frac12(\beta-1)\beta^2 r^{-2}]\\
\geq& -2\tau^2\beta e^{-2r\tau}r^{\beta-1} (\beta+2\tau r)^{-3/2}.
\end{align*}
In addition, we have
\begin{align*}
-2w(r)w'(r)&=e^{-2r\tau}r^{\beta-1}(\beta+2\tau r)^{-1/2}[4\tau^2r-(\beta-1)\beta r^{-1} ]\\
&=e^{-2r\tau}r^{\beta-1}(\beta+2\tau r)^{-3/2}[4\tau^2r(\beta+2\tau r)-(\beta-1)\beta r^{-1}(\beta+2\tau r) ]
\\
&=e^{-2r\tau}r^{\beta-1}(\beta+2\tau r)^{-3/2}[4\tau^2\beta r+8\tau^3 r^2-(\beta-1)\beta r^{-1}(\beta+2\tau r) ]
\\
&\geq 4\tau^2\beta r e^{-2r\tau}r^{\beta-1}(\beta+2\tau r)^{-3/2}.
\end{align*}
Hence,  for $0<\beta\leq1$, we obtain
\begin{align*}
&r\left(w'(r)^2+2w(r)w''(r)\right)-2w(r)w'(r)\geq 2\tau^2\beta r e^{-2r\tau}r^{\beta-1}(\beta+2\tau r)^{-3/2}\geq0.
\end{align*}
Let
\begin{equation}\label{equ:fdef}
  g(r)=\Big(\int_{{\mathbb{S}^{n-1}}}|v(r,\theta)|^2\;d\theta\Big)^\frac12.
\end{equation}
Then, we have
$$g'(r)=\Big(\int_{{\mathbb{S}^{n-1}}}|v(r,\theta)|^2\;d\theta\Big)^{-\frac12}\int_{{\mathbb{S}^{n-1}}}{\rm Re}(v\partial_rv)\;d\theta$$
and
$$|g'(r)|^2\leq\int_{{\mathbb{S}^{n-1}}}|\partial_rv|^2\;d\theta.$$
Next we need to verify the assumption \eqref{equ:funccond} and \eqref{equ:fcond} on $g$.

For the choice of $w$ as in \eqref{equ:wdef} and $g$ as in \eqref{equ:fdef}, we have
\begin{align*}
   & \int_0^\infty \big((w')^2|g|^2+w^2|g'|^2\big)\;dr\\
   \leq &\int_0^\infty  e^{-2r\tau}r^{\beta-1}(\beta+2\tau r)^{-1}[-2\tau^2r+\tfrac12(\beta-1)\beta r^{-1}]^2\int_{{\mathbb{S}^{n-1}}}|v(r,y)|^2\;d\theta\;dr\\
   &+\int_0^\infty  e^{-2r\tau}r^{\beta-1}(\beta+2\tau r)\int_{{\mathbb{S}^{n-1}}}|\partial_rv|^2\;d\theta\;dr.
\end{align*}
Recall that
 $$v(r,y)=r^\frac{n-1}{2} e^{rz}u(r,y),$$
and
$$\partial_rv=r^\frac{n-1}{2}e^{rz}\big(\partial_ru+zu+\tfrac{n-1}{2r}u\big),$$
hence
\begin{align}\nonumber
  & \int_0^\infty \big((w')^2|g|^2+w^2|g'|^2\big)\;dr\\\label{equ:bound1}
  \leq& C(\tau)\int_0^\infty\int_{\mathbb{S}^{n-1}} \Big(\frac{|u|^2}{r^{3-\beta}}+\frac{|\partial_ru|^2}{r^{1-\beta}}\Big)\;d\theta r^{n-1}\;dr\\\label{equ:bound2}
  &+C(\tau)\int_0^\infty\int_{\mathbb{S}^{n-1}} \Big(\frac{|u|^2}{r^{-\beta}}+\frac{|\partial_ru|^2}{r^{-\beta}}\Big)\;d\theta r^{n-1}\;dr.
\end{align}
The boundedness of \eqref{equ:bound1} and \eqref{equ:bound2} follow from \eqref{equ:nabluestlv}. \vspace{0.1cm}

Now we verify \eqref{equ:fcond}. For the choice of  $w$ as in \eqref{equ:wdef} and $g$ as in \eqref{equ:fdef}, we have
 \begin{align*}
   w(r)w'(r)|g(r)|^2 =& e^{-2r\tau}r^{\beta-1}[-2\tau^2r+\tfrac12(\beta-1)\beta r^{-1}]|g(r)|^2.
 \end{align*}
We are reduced to show
that $$\liminf_{r\to0}\tilde{g}(0)=0$$ with
\begin{equation}\label{equ:gdef}
  \tilde{g}(r):=r^\frac{\beta-2}{2}f(r)=r^\frac{\beta-2}{2}\Big(\int_{{\mathbb{S}^{n-1}}}|v(r,y)|^2\;d\theta\Big)^\frac12.
\end{equation}
Recall that
$$v(r,y)=r^\frac{n-1}{2} e^{rz}u(r,y).$$
For the above $\beta\in(\max\{0,1-2\nu_0\},1]$, we can choose $\beta_0\in(\max\{0,1-2\nu_0\},1)$ such that $\beta_0<\beta$, and let $\epsilon=\beta-\beta_0.$

Moreover,  by using \eqref{equ:nabluestlv} with $\beta_0\in(\max\{0,1-2\nu_0\},1)$, we have
\begin{align}\nonumber
\int_0^1\frac{|\tilde{g}(r)|^2}{r^{1+\epsilon}}\;dr=&\int_0^1 r^{\beta-3-\epsilon}\int_{\mathbb{S}^{n-1}}|v(r,y)|^2\;d\theta\;dr\\\nonumber
=&\int_0^1\int_{{\mathbb{S}^{n-1}}}e^{2r\tau}\frac{|u(r,y)|^2}{r^{3-\beta_0}}\;d\theta\;r^{n-1}\;dr\\\label{equ:gzr}
\leq& C(\tau)\int_{r\leq1}\frac{|u|^2}{r^{3-\beta_0}}\;dx<+\infty
\end{align}
which shows that
\begin{equation}\label{equ:g0dui}
  \liminf_{r\to0}\tilde{g}(r)=0,
\end{equation}
 otherwise the integral $\int_0^1\frac{|\tilde{g}(r)|^2}{r^{1+\epsilon}}\;dr$ diverges.

Therefore we have verified the condition of Lemma \ref{lem:w-hardy}. By Lemma \ref{lem:w-hardy}, we obtain
\begin{align*}
  \int_0^\infty\int_{\mathbb{S}^{n-1}} e^{-2r\tau} r^{\beta-1}\left(\beta+ 2r\tau\right) \frac{|v|^2}{r^2}  dr d\theta=&\int_0^\infty w^2 \frac{|g|^2}{r^2}\;dr\\
  \leq&4\int_0^\infty w^2 |g'(r)|^2\;dr\\
  \leq&4\int_0^\infty e^{-2r\tau}r^{\beta-1}(\beta+2\tau r)\int_{{\mathbb{S}^{n-1}}}|\partial_rv|^2\;d\theta\;dr\\
  =&4\int_0^\infty \int_{{\mathbb{S}^{n-1}}}e^{-2r\tau}r^{\beta-1}(\beta+2\tau r)|\partial_rv|^2\;d\theta\;dr,
\end{align*}
which implies \eqref{equ:asmalllv}, and so we conclude the proof of Lemma \ref{lem:asmalllv}.
\end{proof}

\section{The Sobolev inequality and inhomogeneous Strichartz estimate}

In this section,  we prove Theorem \ref{thm-unf-sob} and Theorem \ref{thm-in}.

\subsection{The proof of Theorem \ref{thm-unf-sob}} We set $ \mathcal R_0(\sigma)=(-\Delta-\sigma)^{-1}$ and $\mathcal R(\sigma)=(\LL_V-\sigma)^{-1}$. The proof follows a similar line as in \cite{Miz1} based on the iterated resolvent identity
\begin{align}
\label{proof_theorem_1_2_1}
\mathcal R(\sigma)=\mathcal R_0(\sigma)-\mathcal R_0(\sigma)V\mathcal R_0(\sigma)+\mathcal R_0(\sigma)V\mathcal R(\sigma)V\mathcal R_0(\sigma)
\end{align}
which follows from the standard resolvent formulas
$$
\mathcal R(\sigma)=\mathcal R_0(\sigma)-\mathcal R_0(\sigma)V\mathcal R(\sigma)=\mathcal R_0(\sigma)-\mathcal R(\sigma)V\mathcal R_0(\sigma).
$$
Let $f\in \mathcal C_0^\infty(\R^n)$, $\sigma\in\R^+$ and $(p,q)$ satisfy \eqref{p_q}. Thanks to \eqref{uniform_Sobolev}, it suffices to deal with the second and third terms of the right hand side of \eqref{proof_theorem_1_2_1}.

For the second term, we choose $\widetilde p\ge2$ such that $1/\widetilde p-1/q=2/n$. Since $(\widetilde p,q)$ satisfies \eqref{p_q_0} and $V\in L^{n/2,\infty}$, we can use \eqref{uniform_Sobolev} to obtain
\begin{align*}
\|\mathcal R_0(\sigma)V\mathcal R_0(\sigma)f\|_{L^{q,2}}
&\lesssim \|V\mathcal R_0(\sigma)f\|_{L^{\widetilde p,2}}
\lesssim \|V\|_{L^{n/2,\infty}}\|\mathcal R_0(\sigma)f\|_{L^{q,2}}\\
&\lesssim |\sigma|^{\frac n2(\frac1p-\frac1q)-1}\|f\|_{L^{p,2}}.
\end{align*}

For the third part, we divide the proof into two cases: $1/p-1/q=2/n$ and otherwise.
Let us first suppose $1/p-1/q=2/n$. It is easy to see that
$$
\{(p,q)\ |\ (p,q)\ \text{satisfies}\ \eqref{p_q}\ \text{and}\ 1/p-1/q=2/n\}=\{(p_s,q_s)\ |\ 1-\mu_0<s<1+\mu_0\},
$$
where $p_s=\frac{2n}{n+2(2-s)},q_s=\frac{2n}{n-2s}$. By \eqref{uniform_Sobolev} and H\"older's inequality in Lorentz spaces,
\begin{align*}
\|\mathcal R_0(\sigma)w_1f\|_{L^{q_s,2}}&\lesssim \|w_1f\|_{L^{p_s,2}}\lesssim \|w_1\|_{L^{\frac{n}{2-s},\infty}}\|f\|_{L^2},\\
\|w_2\mathcal R_0(\sigma)f\|_{L^{2}}&\lesssim \|w_2\|_{L^{\frac ns,\infty}}\|\mathcal R_0(\sigma)f\|_{L^{q_s,2}}\lesssim  \|w_2\|_{L^{\frac ns,\infty}}\|f\|_{L^{p_s,2}}
\end{align*}
for all $w_1\in L^{\frac{n}{2-s},\infty}$, $w_2\in L^{\frac ns,\infty}$ and $1/2<s<3/2$. These two estimates, together with \eqref{in-local-sm1} and the fact $r^{-\alpha}\in L^{n/\alpha,\infty}$ and $r^2V\in L^\infty$, imply for $1-\mu_0<s<1+\mu_0$,
\begin{align*}
\|\mathcal R_0(\sigma)V\mathcal R(\sigma)V\mathcal R_0(\sigma)f\|_{L^{q_s,2}}
&\lesssim \|\mathcal R_0(\sigma)r^{-2+s}\|_{L^2\to L^{q_s,2}}\|r^{-s}\mathcal R(\sigma)V\mathcal R_0(\sigma)f\|_{L^2}\\
&\lesssim \|r^{-s}\mathcal R(\sigma)r^{-2+s}\|_{L^2\to L^2}\| r^{-s}\mathcal R_0(\sigma)f\|_{L^2}\\
&\lesssim \|f\|_{L^{p_s,2}},
\end{align*}
which completes the proof of \eqref{unf-sob} for the case when $1/p-1/q=2/n$.

Consider next the case when $2/(n+1)\leq 1/p-1/q< 2/n$. One can find a point $(p_0,q_0)$ satisfying $p_0<p$, $q<q_0$, \eqref{p_q} and $1/p_0-1/q_0=2/n$. Since $(p_0,q)$ and $(p,q_0)$ satisfy \eqref{p_q_0}, \eqref{uniform_Sobolev} and H\"older's inequality then show
\begin{align}\label{equ:es1}
\|\mathcal R_0(\sigma)Vf\|_{L^{q,2}}&\lesssim |\sigma|^{\frac n2(\frac{1}{p_0}-\frac1q)-1}\|Vf\|_{L^{p_0,2}}\lesssim |\sigma|^{\frac n2(\frac{1}{p_0}-\frac1q)-1}\|V\|_{L^{\frac n2,\infty}}\|f\|_{L^{q_0,2}},\\\label{equ:es2}
\|V\mathcal R_0(\sigma)f\|_{L^{p_0,2}}&\lesssim \|V\|_{L^{\frac n2,\infty}}\|\mathcal R_0(\sigma)f\|_{L^{q_0,2}}\lesssim |\sigma|^{\frac n2(\frac{1}{p}-\frac{1}{q_0})-1}\|V\|_{L^{\frac n2,\infty}}\|f\|_{L^{p,2}}.
\end{align}
Since
$$
\frac n2(\frac{1}{p_0}-\frac1q)-1+\frac n2(\frac{1}{p}-\frac1q_0)-1=\frac n2(\frac1p-\frac1q)+\frac n2(\frac{1}{p_0}-\frac{1}{q_0})-2=\frac n2(\frac1p-\frac1q)-1,
$$
the above two estimates \eqref{equ:es1}, \eqref{equ:es2} combined with \eqref{unf-sob} for $(p_0,q_0)$ proved just above, imply
\begin{align*}
\|\mathcal R_0(\sigma)V\mathcal R(\sigma)V\mathcal R_0(\sigma)f\|_{L^{q,2}}
&\lesssim \|R_0(\sigma)V\|_{L^{q_0,2}\to L^{q,2}}\|\mathcal R(\sigma)\|_{L^{p_0,2}\to L^{q_0,2}}\| V\mathcal R_0(\sigma)f\|_{L^{p_0,2}}\\
&\lesssim |\sigma|^{\frac n2(\frac{1}{p}-\frac{1}{q})-1}\|f\|_{L^{p,2}}.
\end{align*}
This completes the proof of Theorem \ref{thm-unf-sob}.

\subsection{The proof of Theorem \ref{thm-in}} We prove Theorem  \ref{thm-in} by using Theorem \ref{thm:w-resolvent} and the iterated Duhamel identity argument in \cite{BM1}.
 Recall $\mathcal{L}_V=-\Delta+V$ and $\mathcal{L}_0=-\Delta$, define the operators
\begin{equation}
\begin{split}
\mathcal{N}_0 F(t)=\int_0^t e^{i(t-s)\mathcal{L}_0} F(s) ds,\quad \mathcal{N} F(t)=\int_0^t e^{i(t-s)\mathcal{L}_V} F(s) ds.
\end{split}
\end{equation}
Setting  $u(t)=e^{i(t-s)\mathcal{L}_V} F(s)$, we can write
$$u(t)=e^{i(t-s)\mathcal{L}_0} F(s)-i\int_s^t e^{i(t-\tau)\mathcal{L}_0} \left(Ve^{-i(\tau-s)\mathcal{L}_V} F(s)\right) d\tau.$$
Integrating in $s\in[0,t]$, we have by Fubini's formula
\begin{equation*}
\begin{split}
 \mathcal{N} F(t)&=\int_0^t e^{i(t-s)\mathcal{L}_V} F(s) ds\\&=\int_0^t e^{i(t-s)\mathcal{L}_0} F(s) ds-i\int_0^t \int_s^t e^{i(t-\tau)\mathcal{L}_0} \left(Ve^{-i(\tau-s)\mathcal{L}_V} F(s)\right) d\tau ds
 \\&=\mathcal{N}_0 F(t)-i\int_0^t \int_0^\tau e^{i(t-\tau)\mathcal{L}_0} \left(Ve^{-i(\tau-s)\mathcal{L}_V} F(s)\right)  ds d\tau
 \\&=\mathcal{N}_0 F(t)-i\int_0^t e^{i(t-\tau)\mathcal{L}_0} \left( V \int_0^\tau e^{-i(\tau-s)\mathcal{L}_V} F(s) ds \right)  d\tau
  \\&=\mathcal{N}_0 F(t)-i\mathcal{N}_0 \left( V (\mathcal{N} F) \right)(t).
\end{split}
\end{equation*}
Therefore \begin{equation}\label{NF}
\begin{split}
 \mathcal{N} F(t)=\mathcal{N}_0 F(t)-i\mathcal{N}_0 \left( V (\mathcal{N} F) \right)(t).
\end{split}
\end{equation}
On the other hand,  by similar argument, we have
$$\mathcal{N}_0 F(t)=\mathcal{N} F(t)+i\mathcal{N} \left( V (\mathcal{N}_0 F) \right)(t),$$ hence $$\mathcal{N} F(t)=\mathcal{N}_0 F-i\mathcal{N} \left( V (\mathcal{N}_0 F) \right)(t).$$
Plugging it into \eqref{NF},  we obtain
\begin{equation*}
\begin{split}
 \mathcal{N} F=\mathcal{N}_0 F-i\mathcal{N}_0 \left( V (\mathcal{N}_0 F) \right)-\mathcal{N}_0 \left( V (\mathcal{N} (V (\mathcal{N}_0 F))) \right),
\end{split}
\end{equation*}
that is
\begin{equation}
\begin{split}
 \mathcal{N} F=\mathcal{N}_0 F-i\left( \mathcal{N}_0 V \mathcal{N}_0  \right)F-\mathcal{N}_0 (V \mathcal{N} V) \mathcal{N}_0 F.
\end{split}
\end{equation}
To prove \eqref{Str-est-in'}, we need to estimate
\begin{equation*}
\begin{split}
&\| \mathcal{N} F\|_{L^2(\R;L^{\frac{2n}{n-2s},2})}\\&
\leq \|\mathcal{N}_0 F\|_{L^2(\R;L^{\frac{2n}{n-2s},2})}+\|\left(\mathcal{N}_0 V \mathcal{N}_0  \right)F\|_{L^2(\R;L^{\frac{2n}{n-2s},2})}+\|\mathcal{N}_0 (V \mathcal{N} V) \mathcal{N}_0 F\|_{L^2(\R;L^{\frac{2n}{n-2s},2})}.
\end{split}
\end{equation*}
By the inhomogeneous Strichartz estimate in
\cite{Foschi,Koh,Vilela,Schippa} for the free Schr\"odinger equation,  we have
\begin{equation}\label{N-0}
\begin{split}
\|\mathcal{N}_0 F\|_{L^2(\R;L^{\frac{2n}{n-2s},2})}\lesssim \|F\|_{L^2(\R;L^{\frac{2n}{n+2(2-s)},2})},\quad \frac{n}{2(n-1)}\leq s\leq\frac{3n-4}{2(n-1)}.
\end{split}
\end{equation}
Since $V=r^{-2}V_0(\theta)$ with $V_0\in\mathcal{C}^\infty({\mathbb{S}^{n-1}})$,  one has $V\in L^{\frac n2,\infty}$. From the Strichartz estimate \eqref{N-0}, we obtain
\begin{equation}\label{N-1}
\begin{split}
&\|\left(\mathcal{N}_0 V \mathcal{N}_0  \right)F\|_{L^2(\R;L^{\frac{2n}{n-2s},2})}\lesssim \|V \mathcal{N}_0 F\|_{L^2(\R;L^{\frac{2n}{n+2(2-s)},2})}\\&\lesssim \|r^{-2}\|_{L^{\frac{n}{2},\infty}}\|\mathcal{N}_0 F\|_{L^2(\R;L^{\frac{2n}{n-2s},2})}
\lesssim \|F\|_{L^2(\R;L^{\frac{2n}{n+2(2-s)},2})}.
\end{split}
\end{equation}

To estimate the final term $\|\mathcal{N}_0 (V \mathcal{N} V) \mathcal{N}_0 F\|_{L^2(\R;L^{\frac{2n}{n-2s},2})}$,
we need two lemmas.
\begin{lemma}\label{lem:e} Let $s\in [n/2(n-1), (3n-4)/2(n-1)]$. Then
\begin{equation}
\|\mathcal{N}_0 r^{-2+s} F\|_{L^2(\R;L^{\frac{2n}{n-2s},2})}\lesssim \|F\|_{L^2(\R;L^{2})},
\end{equation}
and
\begin{equation}
\|r^{-s}\mathcal{N}_0 F\|_{L^2(\R;L^{2})}\lesssim \|F\|_{L^2(\R;L^{\frac{2n}{n+2(2-s)},2})}.
\end{equation}

\end{lemma}
\begin{proof}
It follows from \eqref{N-0} and the H\"older inequality that
\begin{equation}
\begin{split}
\|\mathcal{N}_0 r^{-2+s} F\|_{L^2(\R;L^{\frac{2n}{n-2s},2})}&\lesssim \|r^{-2+s}F\|_{L^2(\R;L^{\frac{2n}{n+2(2-s)},2})}\\&\lesssim \|r^{-2+s}\|_{L^{\frac{n}{2-s},\infty}}\|F\|_{L^2(\R;L^{2})}\lesssim \|F\|_{L^2(\R;L^{2})}
\end{split}
\end{equation}
and
\begin{equation}
\|r^{-s}\mathcal{N}_0 F\|_{L^2(\R;L^{2})}\lesssim \|r^{-s}\|_{L^{\frac{n}{s},\infty}}\|\mathcal{N}_0  F\|_{L^2(\R;L^{\frac{2n}{n-2s},2})}\lesssim \|F\|_{L^2(\R;L^{\frac{2n}{n+2(2-s)},2})}.
\end{equation}
\end{proof}

\begin{lemma}\label{lem:h} Let $\alpha\in R_{\nu_0}$ be defined in \eqref{Rv}, then we have
\begin{equation}\label{in-local-sm}
\|r^{-\alpha}\int_0^t e^{i(t-s)\mathcal{L}_V}r^{-2+\alpha} F ds\|_{L^2(\R;L^{2})}\leq C\|F\|_{L^2(\R;L^{2})}.
\end{equation}
\end{lemma}
\begin{proof} This is a consequence of   D'Ancona's proof \cite{Da} and the weighted resolvent estimate \eqref{in-local-sm1}.
\end{proof}

\vspace{0.2cm}

Note that $r^2V=V_0(\theta)\in L^\infty(\mathbb{S}^{n-1})$. By using Lemma \ref{lem:e} and Lemma \ref{lem:h}, since $s\in A_{\nu_0}$, we prove
\begin{equation}\label{N-2}
\begin{split}
&\|\mathcal{N}_0 (V \mathcal{N} V) \mathcal{N}_0 F\|_{L^2(\R;L^{\frac{2n}{n-2s},2})}\\&=\|(\mathcal{N}_0 r^{-2+s})(r^2V) (r^{-s}\mathcal{N}r^{-2+s})(r^2V)(r^{-s} \mathcal{N}_0) F\|_{L^2(\R;L^{\frac{2n}{n-2s},2})}
\\&\lesssim \| (r^{-s}\mathcal{N}r^{-2+s})(r^2V)(r^{-s} \mathcal{N}_0) F\|_{L^2(\R;L^{2})}\lesssim  \|(r^{-s} \mathcal{N}_0) F\|_{L^2(\R;L^{2})}\\&\lesssim \|F\|_{L^2(\R;L^{\frac{2n}{n+2(2-s)},2})}.\end{split}
\end{equation}
Finally we collect \eqref{N-0}, \eqref{N-1} and \eqref{N-2} to obtain \eqref{Str-est-in'}. Thus we prove Theorem \ref{thm-in}.

\begin{center}

\end{center}

\end{document}